\documentclass[10pt]{article}

\usepackage{hyperref}
\usepackage{amsthm,amsmath,amssymb}
\usepackage{enumerate}
\usepackage{color}

\usepackage{fullpage}
\newtheorem{theorem}{Theorem}
\newtheorem{lemma}[theorem]{Lemma}
\newtheorem{remark}[theorem]{Remark}

\newtheorem{proposition}[theorem]{Proposition}

\newcommand{\id}{\operatorname{I}}

\newcommand{\EE}{\mathbb{E}}

\title{Stochastic proof of the  sharp  symmetrized Talagrand inequality }
\author{Thomas A.~Courtade\footnote{Department of EECS, University of California, Berkeley. Email: courtade@berkeley.edu},   Max Fathi\footnote{Université Paris Cité and Sorbonne Université, CNRS, LJLL and LPSM;  and
  DMA, École Normale Supérieure; and
 Institut Universitaire de France. Email:
mfathi@lpsm.paris}, and Dan Mikulincer\footnote{Department of Mathematics, MIT. Email:  danmiku@mit.edu} }%
\date{\today}

\begin{document}

\maketitle

\begin{abstract}
We give a new proof of the sharp symmetrized form of Talagrand's transport-entropy inequality.  Compared to stochastic proofs of other Gaussian functional inequalities, the new idea here is a certain  coupling induced by time-reversed martingale representations. 
\end{abstract}


\section{Introduction}

The goal of this note is to give a  short  stochastic proof of the following symmetrized Talagrand inequality:
\begin{theorem}[{\cite[Th.~1.1]{MF}}]\label{thm:symmetrizedTalagrand}
For  probability measures $\mu,\nu$ on $\mathbb{R}^n$ with finite second moments and $\mu$ centered,  
\begin{align} 
W_2(\mu,\nu)^2 \leq 2 D(\mu\|\gamma)  + 2 D(\nu\|\gamma), \label{eq:symmTal}
\end{align}
where $W_2$ is   2-Wasserstein distance, $D$ is  relative entropy, and $\gamma$ is the standard Gaussian measure on $\mathbb{R}^n$. 
\end{theorem}
By duality, \eqref{eq:symmTal} is  formally equivalent to the functional Blaschke--Santal\'o inequality \cite[Theorem 1.2]{LBS}, which states that if Borel functions $f,g : \mathbb{R}^n \to \mathbb{R}$ satisfy 
\begin{align}
\int_{\mathbb{R}^n} xe^{-f(x)}dx = 0 ~~~\mbox{and}~~~f(x) + g(y) \geq \langle x, y\rangle, ~~\forall x,y\in \mathbb{R}^n,  \notag
\end{align}
then 
\begin{align} 
\left(\int_{\mathbb{R}^n} {e^{-f(x)}}dx\right)\left(\int_{\mathbb{R}^n} {e^{-g(x)}dx}\right) \leq (2\pi)^n. \label{functional_santalo}
\end{align}
Equality holds for quadratic $f$, and $g = f^*$, its Legendre dual.   Despite the  equivalence, \eqref{eq:symmTal} may be regarded as a formal strengthening of \eqref{functional_santalo} in the sense that \eqref{functional_santalo} is recovered from Theorem \ref{thm:symmetrizedTalagrand} by weak duality: briefly,  for $f,g$ satisfying the hypotheses demanded by \eqref{functional_santalo},   take $d\mu(x) \propto e^{-f(x)} dx$ and  $d\nu(x) \propto e^{-g(x)} dx$  in \eqref{eq:symmTal} and simplify to obtain \eqref{functional_santalo}.  The reverse implication corresponds to strong duality, and is more difficult. See \cite{MF}.

Inequality  \eqref{functional_santalo} is a functional generalization of the earlier Blaschke--Santal\'o inequality for the volume product of convex sets. It was proved in \cite{AKM} (and earlier in K. Ball's PhD thesis \cite{Ball} in a restricted setting of even functions). The original proof relied on the usual Blaschke--Santal\'o inequality applied to level sets. Lehec later gave two alternative proofs; one using induction on the dimension \cite{LBS}, and the other \cite{Leh} using the Prek\'opa--Leindler inequality and the Yao--Yao partition theorem. This last proof actually yields a more general statement, originally due to \cite{FM}, but the present work shall be restricted to the classical setting. More recently, a new semigroup proof of the inequality for even functions was established in \cite{NT} using improved hypercontractive estimates for the heat flow, and then simplified in \cite{CGNT}. Let us also mention a recent generalization to several functions under a symmetry assumption, due to Kolesnikov and Werner \cite{KW}.

Equivalence between integral inequalities of the form \eqref{functional_santalo} and transport inequalities of the form \eqref{eq:symmTal} via duality goes back to \cite{BG}, where they studied Talagrand quadratic transport-entropy inequality \cite{Tal} (which is \eqref{eq:symmTal} in the particular case $\mu = \gamma$). Duality for transport inequalities involving three measures, such as \eqref{eq:symmTal}, was first considered in \cite[Proposition 8.2]{GL}.

 Stochastic proofs of functional inequalities, in particular using Brownian motion and Girsanov's theorem, go back to Borell's stochastic proof of the Prek\'opa--Leindler inequality \cite{Bor}. Our present work is motivated by Lehec's short stochastic proofs of various functional inequalities \cite{L}, including in particular Talagrand's transport-entropy inequality.

\smallskip

{\bf Acknowledgments.} T.C.~acknowledges NSF-CCF 1750430, the hospitality of the Fondation Sciences Mathématiques de Paris (FSMP) and the LPSM at the Université Paris Cité. M.F.~was supported by the Agence Nationale de la Recherche (ANR) Grant ANR-23-CE40-0003 (Project CONVIVIALITY). D.M.~was partially supported by a Simons Investigator Award 622132.

\section{The Stochastic Proof}

We'll work on   the  Wiener space $(\Omega, \mathcal{B}, \mathbb{P})$, where $\Omega$ is the set of continuous paths $\omega : [0,1]\to \mathbb{R}^n$ starting at 0, $\mathcal{B}$ is the usual Borel $\sigma$-algebra, and $\mathbb{P}$ is the Wiener measure.  Let $B_t(\omega) := \omega(t)$ be the coordinate process, so that $B = (B_t)_{0\leq t \leq 1}$ is a standard Brownian motion, and so is the time-reversed process $\hat{B}_t := B_1 - B_{1-t}$.    Let $\mathcal{F} = (\mathcal{F}_t)_{0 \leq t \leq 1}$ and $\mathcal{F}^+ = (\mathcal{F}^+_{t})_{0 \leq t \leq 1}$ denote the filtrations generated by $B$ and $\hat{B}$, respectively.  For each $t\in [0,1]$,  $\mathcal{F}_t$ and $\mathcal{F}^+_{1-t}$ are complementary, in the sense that they are independent and $\mathcal{B} =\sigma( \mathcal{F}_t \cup \mathcal{F}_{1-t}^+)$.  Henceforth, $\|\cdot\|$ denotes the $\ell^2$ norm, and $\id$ denotes the identity matrix.

Apart from standard facts in stochastic calculus, we'll need two lemmas.  The first is a  variational representation of entropy, obtained as a consequence of Girsanov's theorem; it has previously been  applied to study rigidity and stability of various functional inequalities (see, e.g.,  \cite{ACZ, EM, M}).
\begin{lemma}[\cite{EM}]\label{lem:varRep}
For  a centered probability measure $\mu$  on $\mathbb{R}^n$ with finite second moments, we have 
\begin{align}
D(\mu \| \gamma) = \inf_{F} \frac{1}{2}\int_{0}^1 \frac{\EE[\|F_t - \id\|^2 ]}{1-t}dt, \label{repEntropy}%
\end{align}
where the infimum is over  all $\mathcal{F}$-adapted matrix-valued processes $F = (F_t)_{0 \leq t \leq 1}$ such that $\int_0^1 F_t dB_t \sim \mu$.  
\end{lemma}
\begin{remark}
By symmetry, the same representation holds if we consider $\mathcal{F}^+$-adapted   $F$ with $\int_0^1 F_t d\hat{B}_t \sim \mu$.
\end{remark}
Stochastic proofs of several other  functional inequalities use representation formulas for the entropy and linear couplings of Brownian motions (cf.~\cite{EM,L}). Our proof will similarly rely on the  representation formula \eqref{repEntropy} for the entropy\footnote{Note that the representation \eqref{repEntropy}    is not the same as that used in \cite{L}. However, it is derived from \cite[Theorem 4]{L} by combination with the martingale representation theorem.}, but makes use of a  new coupling induced by time-reversal.    The next lemma is the crucial new ingredient; it relates   martingale representations in terms of $(B_t)_{0\leq t \leq 1}$ and its time-reversal $(\hat{B}_t)_{0\leq t \leq 1}$. 
\begin{lemma}\label{lem:integralInequality}
If $X \in L^2(\Omega, \mathcal{B}, \mathbb{P})$ is a zero-mean $\mathbb{R}^n$-valued random vector with    martingale representations
$$
X = \int_0^1 F_t dB_t = \int_{0}^1 G_t d\hat{B}_t,
$$
then 
\begin{align}
   \int_{t}^1 \EE[\|F_s-\id\|^2]ds \geq  \int_0^{1-t} \EE[\|G_{s}-\id \|^2]ds , ~~~\forall 0 \leq t \leq 1.\label{eq:tailBounds}
\end{align}  
~
\end{lemma}
\begin{proof}
By the Pythagorean theorem, convexity, and  independence of $\mathcal{F}_t$ and $\mathcal{F}^+_{1-t}$, we have 
\begin{align}
\EE[ \|X\|^2] - \EE[\|\EE[X|\mathcal{F}_t]\|^2] = \EE[\|X - \EE[X|\mathcal{F}_{t}]\|^2]   \geq    \EE[\|\EE[X- \EE[X|\mathcal{F}_{t}]|\mathcal{F}^+_{1-t}]\|^2]= \EE[\|\EE[X|\mathcal{F}^+_{1-t}]\|^2].  \label{eq:Pythagorean}
\end{align}
Since $\EE[X|\mathcal{F}_t] = \int_0^t F_s dB_s$ and $\EE[X|\mathcal{F}^+_{1-t}] = \int_0^{1-t} G_s d\hat{B}_s$, three applications of It\^o's isometry give 
$$
     \int_{t}^1 \EE[\|F_s\|^2]ds = \EE[ \|X\|^2] - \EE[\|\EE[X|\mathcal{F}_t]\|^2] \geq  \EE[\|\EE[X|\mathcal{F}^+_{1-t}]\|^2] = \int_0^{1-t} \EE[\|G_s\|^2]ds  , ~~~\forall 0 \leq t \leq 1. 
 $$
For the moment, assume $X$ is differentiable in the Malliavin sense.  If $D_t$ denotes the usual Malliavin derivative, and  $\hat{D}_t$ denotes the Malliavin derivative with respect to the time-reversed path space, then $D_t X = \hat{D}_{1-t} X$. %
Thus,  a consequence of the Clark--Ocone theorem and iterated expectation is the identity
 $$
 \EE[F_t] = \EE[ \EE[D_t X|\mathcal{F}_t] ] = \EE[ D_t X ]  = \EE[\hat{D}_{1-t}X]=  \EE[ \EE[\hat{D}_{1-t} X|\mathcal{F}^+_{1-t}] ]  = \EE[ G_{1-t} ], ~~~\forall 0 \leq t \leq 1.
 $$
Combining the previous two  observations   gives  \eqref{eq:tailBounds}.  
 If $X$ is not Malliavin-differentiable,  then by density of $\mathbb{D}^{1,2}$ in $L^2$ we may approximate it in $L^2$ by such a sequence, and use It\^o's isometry to conclude $L^2$-convergence of the coefficients in the martingale representations. Thus,  the desired conclusion holds generally.   
\end{proof}

\begin{proof}[Proof of Theorem \ref{thm:symmetrizedTalagrand}] The inequality is invariant with respect to translations of $\nu$, so we may also assume $\nu$ is  centered.  
 Let $F = (F_t)_{0 \leq t \leq 1}$ be any $\mathcal{F}$-adapted process  such that $\int_0^1 F_t dB_t \sim \mu$, and let $H = (H_t)_{0 \leq t \leq 1}$ be any $\mathcal{F}^+$-adapted process such that $\int_0^1 H_t d\hat{B}_t \sim \nu$.   Let $G = (G_t)_{0 \leq t \leq 1}$ be the martingale representation of  $\int_0^1 F_t dB_t$ in terms of the time-reversed Brownian motion $\hat{B}$; i.e., $G$ is $\mathcal{F}^+$-adapted, satisfying $\int_0^1 G_t d\hat{B}_t =\int_0^1 F_t dB_t \sim \mu$.   By the Tonelli theorem and Lemma \ref{lem:integralInequality}, we have the estimate
      \begin{align*}
      \int_0^1 \frac{\EE[\|G_{s}-\id\|^2]}{s} ds&= \int_0^1  \EE[\|G_{s}-\id \|^2] ds +  \int_0^1 \frac{1}{(1-t)^2} \left( \int_0^{1-t} \EE[\|G_{s}-\id \|^2] ds \right) dt\\
      &\leq \int_0^1  \EE[\|F_s-\id \|^2] ds +  \int_0^1 \frac{1}{(1-t)^2} \left( \int_t^1 \EE[\|F_s-\id \|^2] ds \right) dt = \int_0^1 \frac{\EE[\|F_s-\id \|^2]}{1-s}ds.
   \end{align*}
Taking this together with the definition of $W_2$, It\^o's isometry,   and convexity of $\|\cdot\|^2$, gives
\begin{align*}
W_2(\mu,\nu)^2 \leq \EE \left\| \int_{0}^1 (G_t -H_t)d\hat{B}_t  \right\|^2 
=  \int_{0}^1 \EE \|G_t -H_{t}\|^2 dt
 &\leq \int_{0}^1 \frac{\EE \|G_t - \id\|^2}{t} dt + \int_{0}^1 \frac{\EE \|H_{t} - \id\|^2}{1-t}dt
 \\ &\leq   \int_0^1 \frac{\EE[\|F_t-\id \|^2]}{1-t}dt+  \int_{0}^1 \frac{\EE \|H_t - \id\|^2}{1-t} dt.
\end{align*}
With the help of  Lemma  \ref{lem:varRep}, optimizing over $F$ and $H$  completes the proof . \end{proof}

\section{Remarks on the  Approach}
\noindent{\bf Equality cases:} The  equality cases for \eqref{eq:symmTal} are also evident from the given proof.  Indeed, if $D(\mu\|\gamma) <\infty$, then the infimum in \eqref{repEntropy} is a.s.-uniquely achieved by an $\mathcal{F}$-adapted process $F = (F_t)_{0 \leq t \leq 1}$.  Defining $X := \int_0^1 F_t dB_t\sim \mu$ for this  particular $F$, equality in \eqref{eq:symmTal} implies equality in \eqref{eq:Pythagorean} for a.e.~$t\in[0,1]$, which means that $X - \EE[X|\mathcal{F}_t]$ is $\mathcal{F}^+_{1-t}$-measurable for a.e.~$t\in [0,1]$.  By Proposition \ref{prop:Gaussian} below,  this ensures $X\sim \mu$ is Gaussian.  By symmetry, any extremal $\nu$ is also Gaussian, and  explicit computation shows that  $\mu,\nu$ are extremizers in \eqref{eq:symmTal} iff   $\mu = N(0,C)$ and $\nu = N(\theta,C^{-1})$ for  some $\theta\in \mathbb{R}^n$ and positive  definite $C\in \mathbb{R}^{n\times n}$.  
\begin{proposition}\label{prop:Gaussian}
Let $X\in L^2(\Omega, \mathcal{B},\mathbb{P})$ admit martingale representation $X = \int_0^1 F_t dB_t$.  If $X - \EE[X|\mathcal{F}_t]$ is $\mathcal{F}^+_{1-t}$-measurable for a.e.~$t\in [0,1]$, then $X$ is Gaussian. 
\end{proposition}
\begin{proof}
Define $M_t := \int_{0}^t F_s dB_s$.  The hypothesis is equivalent to requiring that $(M_1 - M_t)$ is $\mathcal{F}^+_{1-t}$-measurable for each $t\in \mathcal{D}$, where $\mathcal{D}$ is dense in $[0,1]$.  Fix  any $s,t\in \mathcal{D}$, with $s\leq t$.  Since  $(M_1 - M_t)$ is $\mathcal{F}^+_{1-t}$-measurable by hypothesis, and $(M_t - M_s)$ is $\mathcal{F}_{t}$-measurable by definition,  complementarity ensures $(M_1 - M_t)$ and $(M_t - M_s)$ are independent.  Iterating this procedure on the $(M_1 - M_t)$ term allows us to conclude that $(M_t)_{0 \leq t \leq 1}$ has independent increments, provided the endpoints of the increments are in $\mathcal{D}$.  Since $X\in L^2(\Omega, \mathcal{B},\mathbb{P})$, a version of $(M_t)_{0 \leq t \leq 1}$ admits continuous  sample paths, and we conclude by density of $\mathcal{D}$ that $(M_t)_{0 \leq t \leq 1}$ has (square-integrable) independent increments generally, and is thus a Gaussian process.
\end{proof}

\noindent{\bf Importance of the coupling induced by time-reversal:}~With the proof of Theorem \ref{thm:symmetrizedTalagrand} in hand and the equality cases characterized, we  highlight the importance of the coupling based on time-reversal. As is  the case in previous stochastic proofs of functional inequalities, one could  appeal to  martingale representations   $\int_0^1 F_t dB^1_t \sim \mu$ and $\int_0^1 G_t dB^2_t \sim \nu$ with linearly coupled Brownian motions $B^1$ and $B^2$ (or, equivalently, Brownian motions $B^1$ and $B^2$ adapted to a common filtration) to couple $\mu$ and $\nu$.   This approach cannot work, as we now explain. 	

	Working in dimension $n=1$ for simplicity, recall that when $\mu = N(0,\alpha)$ with $\alpha>0$, the minimizer $F$ in \eqref{repEntropy} has an  explicit expression  (e.g., \cite[Section 2]{EM}).  In particular, we have the implication
			\begin{align*}
	\int F_t dB^1_t \sim \mu \mbox{~and~} D(\mu \| \gamma) =   \frac{1}{2}\int_{0}^1 \frac{\EE[\|F_t - \id\|^2 ]}{1-t}dt ~~\Rightarrow~~  	F_t = \frac{\alpha}{1-t + \alpha t}.
	\end{align*}
Likewise, for $\nu = N(0,\alpha^{-1})$,  the ``optimal"   representation of $\nu$ with respect to $B^2$ satisfies   
	$$
 \int_0^1 G_t dB^2_t \sim \nu \mbox{~and~} D(\nu \| \gamma) =   \frac{1}{2}\int_{0}^1 \frac{\EE[\|G_t - \id\|^2 ]}{1-t}dt ~~\Rightarrow~~ G_t = \frac{1}{ \alpha(1-t) +  t}.
$$ 
Since $B^1$ and $B^2$ are linearly coupled standard Brownian motions, we can write 
\begin{align}
\begin{bmatrix}
B^1_t \\ B^2_t
\end{bmatrix} = \begin{bmatrix}
1 & \sigma \\ \sigma & 1
\end{bmatrix}^{1/2} B_t, ~~~0 \leq t \leq 1,\label{linCoupling}
\end{align}
for some $|\sigma| \leq 1$, where $(B_t)_{0 \leq t \leq 1}$ is a 2-dimensional standard Brownian motion.  Thus, our construction induces a  coupling $\pi_{\sigma}$  of $X := \int_0^1 F_t dB^1_t\sim \mu$ and $Y := \int_0^1 G_t dB^1_t\sim \nu$, depending on $\sigma$, which satisfies 
\begin{align*}
\EE_{\pi_{\sigma}} \|X-Y\|^2 = \EE \left\| \int_0^1 F_t dB^1_t- \int_0^1 G_t dB^2_t  \right\|^2 &= \int_0^1 \left(  F_t^2 + G_t^2 - 2 \sigma F_t G_t \right) dt \\
&= \begin{cases}
 \left(\alpha + \frac{1}{\alpha}\right)  -   \frac{4 \sigma \log \alpha}{ \left(\alpha - \frac{1}{\alpha}\right) } & \mbox{if~}\alpha \neq 1\\
 2(1-\sigma) & \mbox{if~}\alpha = 1,
 \end{cases}
\end{align*}
where we made use of It\^o's isometry and \eqref{linCoupling}.  A simple calculation  reveals that 
$$
\min_{\sigma : |\sigma|\leq 1} \EE_{\pi_{\sigma}} \|X-Y\|^2 \geq \alpha + \frac{1}{\alpha}  - 2 = W_2(\mu,\nu)^2 = 2 D(\mu\|\gamma) + 2 D(\nu\|\gamma), 
$$
with equality if and only if $\alpha = 1$.  So, with the exception of  the trivial case $\mu = \nu = \gamma$, the established stochastic approach to proving functional inequalities using linearly coupled Brownian motions fails to produce the requisite optimal coupling between $\mu$ and $\nu$ in all extremal cases (at least, in this implementation).  This suggests that coupling through time-reversal lends a useful new degree of freedom to the stochastic program for proving functional inequalities.

\end{document}